\documentclass[12pt]{amsart}
\usepackage{amsthm}
\usepackage{amsmath}
\usepackage[margin=1.25in]{geometry}
\usepackage{amssymb}
\usepackage{verbatim}

\theoremstyle{plain}
\newtheorem{theorem}{Theorem}[section]
\newtheorem*{nonumbertheorem}{Theorem}
\newtheorem*{nonumbercorollary}{Corollary}

\newtheorem{main}{Theorem}

\newtheorem{lemma}[theorem]{Lemma}
\newtheorem{proposition}[theorem]{Proposition}

\newtheorem*{claim}{Claim}

\theoremstyle{definition}
\newtheorem*{nonumberdefinition}{Definition}
\newtheorem{definition}[theorem]{Definition}

\theoremstyle{remark}

\numberwithin{equation}{section}
\newcommand{\thm}[1]{Theorem \ref{#1}}\newcommand{\thms}[2]{Theorems \ref{#1} and \ref{#2}}
\newcommand{\Sec}[1]{Section \ref{#1}}
\newcommand{\lem}[1]{Lemma \ref{#1}}
\newcommand{\lems}[2]{Lemmas \ref{#1} and \ref{#2}}

\newcommand{\prop}[1]{Proposition \ref{#1}}
\newcommand{\TABLE}[1]{Table \ref{#1}}\newcommand{\TABLES}[2]{Tables \ref{#1} and \ref{#2}}

\newcommand{\Z}{\mathbb{Z}}
\newcommand{\Q}{\mathbb{Q}}

\newcommand{\C}{\mathbb{C}}
\newcommand{\HH}{\mathbb{H}}
\newcommand{\embedded}{\hookrightarrow}
\DeclareMathOperator{\cod}{cod}
\DeclareMathOperator{\diag}{diag}

\newcommand{\ceil}[1]{\left\lceil #1 \right\rceil}
\newcommand{\floor}[1]{\left\lfloor #1 \right\rfloor}

\newcommand{\of}[1]{\left(#1\right)}

\newcommand{\propI}{Property $\mathcal{I}$ }

\newcommand{\tensor}{\otimes}

\newcommand{\id}{\mathrm{id}}
\DeclareMathOperator{\rank}{rank}

\DeclareMathOperator{\dk}{dk}

\author{Lee Kennard}
\date{} 
\title[Positively curved metrics with large symmetry rank]{Positively curved metrics on symmetric spaces with large symmetry rank}

\begin{document}

\begin{abstract}
We prove an obstruction at the level of rational cohomology in small degrees to the existence of positively curved metrics with large symmetry rank. The symmetry rank bound is logarithmic in the dimension of the manifold. As an application, we provide evidence for a generalized conjecture of Hopf that says that no symmetric space of rank at least two admits a metric with positive curvature.
\end{abstract}
\maketitle


A well known conjecture of Hopf states that $S^2\times S^2$ admits no metric of positive sectional curvature. More generally, one might ask whether any nontrivial product manifold admits a metric with positive sectional curvature.

Another way to generalize this conjecture of Hopf is to observe that $S^2\times S^2$ is a rank two symmetric space. While the compact, simply connected rank one symmetric spaces, i.e., $S^n$, $\C P^n$, $\HH P^n$, and $\mathrm{Ca}P^2$, admit metrics with positive sectional curvature, it is conjectured that no symmetric space of rank greater than one admits such a metric. Our first main result provides evidence for this conjecture under the assumption of symmetry:

\smallskip
\begin{main}\label{symspsTHM}
Suppose $M^n$ has the rational cohomology of a simply connected, compact symmetric space $N$. If $M$ admits a positively curved metric with symmetry rank at least $2\log_2 n+7$, then $N$ is a product of spheres times either a rank one symmetric space or a rank $p$ Grassmannian $SO(p+q)/SO(p)\times SO(q)$ with $p\in\{2,3\}$.
\end{main}
\smallskip

Recall that the symmetry rank is defined as the rank of the isometry group. The assumption that the symmetry rank is at least $r$ is equivalent to the existence of an effective, isometric $T^r$-action on $M$. 

If we restrict to the case where $N$ is an irreducible symmetric space, then product manifolds are excluded, so the only possibilities are that $N$ has rank one or that $N$ is a rank two or rank three real Grassmannian. See \thm{thmSymSps} for a more detailed statement. For example, $N$ cannot be $S^k\times S^{n-k}$ with $1 < k < 16$.

One might compare this result to Theorem 5 in \cite{Wilking03}, which implies that, for a field $k$ and a closed, simply connected $n$-manifold $M$ with positive sectional curvature, if $M$ has $n\geq 6000$ and symmetry rank at least $\frac{n}{6} + 1$, then $H^*(M;k)$ is isomorphic to that of a rank one symmetric space, $S^2\times\HH P^{(n-2)/4}$, or $S^3\times\HH P^{(n-3)/4}$.

The obstruction (see \thm{thmP-intro}) we prove in order to obtain \thm{symspsTHM} is at the level of rational cohomology in small degrees. Since taking products with spheres does not affect cohomology in small degrees, and since the Grassmannians $SO(2+q)/SO(2)\times SO(q)$ and $SO(3+q)/SO(3)\times SO(q)$ have the same rational cohomology ring in small degrees as the complex and quaternionic projective spaces, respectively, our methods cannot exclude them.

\smallskip

The second main result is related to the Bott conjecture and a second conjecture of Hopf. Recall that the Bott conjecture states that a nonnegatively curved manifold is rationally elliptic, which, in paticular, implies that the Euler characteristic is positive if and only if the odd Betti numbers vanish. The conjecture of Hopf states that the Euler characteristic of an even-dimensional, positively curved manifold is positive. Hence the conjectures together would imply that even-dimensional, positively curved manifolds have vanishing odd Betti numbers. The second part of the following theorem provides some evidence for this statement:

\smallskip
\begin{main}\label{corBott}
Assume $n\geq c\geq 2$, and let $M^n$ be a closed, positively curved manifold with symmetry rank at least $2\log_2(n) + \frac{c}{2} - 1$. Then 	\begin{itemize}
	\item The Betti numbers $b_{2i}(M)$ for $2i<c$ agree with those of $S^n$, $\C P^{n/2}$, or $\HH P^{n/4}$.
	\item If $n\equiv 0\bmod{4}$, then the rational cohomology of $M$ in degrees less than $c$ agrees with that of $S^n$, $\C P^{n/2}$, or $\HH P^{n/4}$. In particular, $b_{2i+1}(M) = 0$ for $2i+1 < c$.
  \end{itemize}
\end{main}
\smallskip





In order to explain our main topological result, which is the crucial step in proving \thms{symspsTHM}{corBott}, we make the following definition:

\smallskip
\begin{nonumberdefinition}
For a closed, orientable manifold $M$, we say that $H^*(M;\Q)$ is $4$-periodic up to degree $c$ if there exists $x\in H^4(M;\Q)$ such that the map $H^i(M;\Q)\to H^{i+4}(M;\Q)$ given by $y\mapsto xy$ is a surjection for $i=0$, an isomorphism for $0<i<c-4$, and an injection for $i=c-4$. If $M$ has dimension $n$ and $H^*(M;\Q)$ is $4$-periodic up to degree $n$, we simply say that $H^*(M;\Q)$ is $4$-periodic.
\end{nonumberdefinition}
\smallskip

In particular, if $x \neq 0$ in the definition, then $H^{4s}(M;\Q) \cong \Q$ for $0\leq s < \frac{c}{4}$. However we abuse notation slightly by allowing $x=0$, hence we say that a rationally $(c-1)$-connected space has 4-periodic rational cohomology up to degree $c$.

Examples of $n$-manifolds with $4$-periodic rational cohomology are $S^n$, $\C P^{n/2}$, $\HH P^{n/4}$, $S^2\times \HH P^{(n-2)/4}$, and $S^3\times \HH P^{(n-3)/4}$. 
By taking a product of one of these spaces with any rationally $(c-1)$-connected space, we obtain examples of spaces with $4$-periodic rational cohomology up to degree $c$.

The main cohomological obstruction to the existence of positively curved metrics with large symmetry rank can now be stated as follows:

\smallskip
\begin{main}\label{thmP-intro}
Let $n\geq c\geq 2$. If $M$ is a closed, simply connected, positively curved $n$-manifold with symmetry rank at least $2\log_2(n) + \frac{c}{2}-1$, then $H^*(M;\Q)$ is $4$-periodic up to degree $c$.
\end{main}
\smallskip

Taking $c = 16$ and restricting to the situation where $M$ has the rational cohomology of a compact symmetric space, we obtain \thm{symspsTHM} as a corollary by comparing this obstruction with the classification of symmetric spaces. See \Sec{symsps} for details.

Another, more immediate consequence of \thm{thmP-intro} follows by taking $c=10$ and concluding that the Betti numbers of $M$ satisfy $b_4(M) = b_8(M) \leq 1$:

\smallskip
\begin{nonumbercorollary}\label{corCheeger}
If $M^n$ is the connected sum of $M',M''\in\{\C P^{n/2}, \HH P^{n/4}, \mathrm{Ca}P^2\}$, then $M$ does not admit a positively curved metric with symmetry rank at least $2\log_2 (4n)$.
\end{nonumbercorollary}
\smallskip

On the other hand, recall that such connected sums admit metrics, called Cheeger metrics, with nonnegative curvature (see \cite{Cheeger73}).

A final corollary relates to a conjecture of Chern, which states that, for a positively curved manifold, every abelian subgroup of the fundamental group is cyclic. Examples given in \cite{Shankar98, Bazaikin99, Grove-Shankar00} show that this conjecture is false. However modified versions of the Chern conjecture have been verified under the addtional assumption that the symmetry rank is at least a linear function of the dimension (see \cite{Rong02, Wilking03, Frank-Rong-Wang??pre, Wang07}).

\smallskip
\begin{nonumbercorollary}\label{corChern}
If $M^{4n+1}$ is a closed, positively curved manifold with symmetry rank $2\log_2(4n+1)$, then $\pi_1(M)$ is isomorphic to the fundamental group of a rational homology sphere $\Sigma^{4m+1}$ that admits positive sectional curvature.
\end{nonumbercorollary}
\smallskip

We use many ideas from \cite{Wilking03, Kennard12pre}, including Wilking's connectedness theorem in \cite{Wilking03} and the periodicity theorem in \cite{Kennard12pre}. These theorems place restrictions on the cohomology of a closed, positively curved manifold in the presence of totally geodesic submanifolds of small codimension (see \Sec{Preliminaries}). Since fixed-point sets of isometries are totally geodesic, these become powerful tools in the presence of symmetry. The connection made in \cite{Wilking03} to the theory of error-correcting codes also plays a role. Here we will use the Griesmer bound, which is well suited to the logarithmic symmetry rank bound with which we are working.

The proofs of our results involve a stronger version of \thm{thmP-intro}, namely \thm{thmP}. It states that, under the assumptions of \thm{thmP-intro}, there exists a $c$-connected inclusion $N \subseteq M$ of a compact submanifold $N$ such that $H^*(N;\Q)$ is $4$-periodic. Moreover, one can ensure that $\dim N \equiv \dim M \bmod{4}$. The advantage of this statement is that one can apply Poincar\'e duality to conclude the following about $N$:
	\begin{enumerate}
	\item the subring of $H^*(N;\Q)$ made up of elements of even degree is isomorphic to that of $S^m$, $\C P^m$, $\HH P^{m/4}$, or $S^2\times \HH P^{(n-2)/4}$, and
	\item if $\dim M \equiv 0 \bmod{4}$, then $N$ has vanishing odd-dimensional cohomology.
	\end{enumerate}
For $i<c$, the map $H^i(M;\Q) \to H^i(N;\Q)$ induced by inclusion is an isomorphism, so one can use these observations to conclude \thm{corBott}. The corollary above about the Chern conjecture follows in a similar way but is a little more involved. It is proven in the discussion following \thm{thmP}.

\smallskip
This paper is organized as follows. In \Sec{Preliminaries}, we quote some preliminary results and prove a lemma using the Griesmer bound. In \Sec{P}, we prove \thms{thmP}{thmP-intro}. In \Sec{symsps}, we study the topological obstructions imposed by \thm{thmP-intro} and prove \thm{symspsTHM}.

\subsection*{Acknowledgements}
This work is part of the author's Ph.D. thesis. The author would like to thank his advisor, Wolfgang Ziller, for helpful comments and for suggesting this line of work. The author would also like to thank Anand Dessai for useful discussions.

\bigskip
\section{Preliminaries and the Griesmer bound}\label{Preliminaries}
\bigskip

An important result for this work is Wilking's connectedness theorem:

\begin{theorem}[Connectedness Theorem, \cite{Wilking03}]\label{THMconnectednesstheorem}
Suppose $M^n$ is a closed Riemannian manifold with positive sectional curvature.
 \begin{enumerate}
  \item If $N^{n-k}$ is connected and totally geodesic in $M$, then $N\embedded M$ is $(n- 2k + 1)$-connected. 
  \item If $N_1^{n-k_1}$ and $N_2^{n-k_2}$ are totally geodesic with $k_1\leq k_2$, then $N_1\cap N_2\embedded N_2$ is $(n - k_1 - k_2)$-connected.
 \end{enumerate}
\end{theorem}

Recall an inclusion $N\embedded M$ is called $h$-connected if $\pi_i(M,N)=0$ for all $i\leq h$. It follows from the relative Hurewicz theorem that the induced map $H_i(N;\Z) \to H_i(M;\Z)$ is an isomorphism for $i<h$ and a surjection for $i=h$. The following is a topological consequence of highly connected inclusions of closed, orientable manifolds:

\begin{theorem}[\cite{Wilking03}]
Let $M^n$ and $N^{n-k}$ be closed, orientable manifolds. If $N\embedded M$ is $(n-k-l)$-connected with $n-k-2l>0$, then there exists $e\in H^k(M;\Z)$ such that the maps $H^i(M;\Z)\to H^{i+k}(M;\Z)$ given by $x\mapsto ex$ are surjective for $l\leq i<n-k-l$ and injective for $l<i\leq n-k-l$.
\end{theorem}

In particular, in the case where $l = 0$, the integral cohomology of $M$ is $k$-periodic according to the following definition:

\begin{definition}
For a space $M$, a coefficient ring $R$, and a positive integer $c$, we say that $H^*(M;R)$ is $k$-periodic up to degree $c$ if there exists $x\in H^k(M;R)$ such that the map $H^i(M;R)\to H^{i+k}(M;R)$ given by $y\mapsto xy$ is a surjection for $i=0$, an isomorphism for $0 < i < c-k$ and an injection for $i = c-k$.

If, in addition, $M$ is a $c$-dimensional, closed, $R$-orientable manifold, then we say that $H^*(M;R)$ is $k$-periodic.
\end{definition}

In \cite{Kennard12pre}, the action of the Steenrod algebra was exploited to prove the following:

\begin{theorem}[Periodicity Theorem, \cite{Kennard12pre}]\label{THMstrongperiodicity}
Let $N^n$ be a closed, simply connected Riemannian manifold with positive sectional curvature. Let $N_1^{n-k_1}$ and $N_2^{n-k_2}$ be totally geodesic, transversely intersecting submanifolds.
	\begin{enumerate}
	\item If $2k_1 + 2k_2 \leq n$, then the rational cohomology rings of ${N}$, ${N_1}$, ${N_2}$, and ${N_1\cap N_2}$ are $\gcd(4,k_1)$-periodic.
	\item If $3k_1 + k_2 \leq n$, and if $N_2$ is simply connected, then the rational cohomology rings of ${N_2}$ and ${N_1\cap N_2}$ are $\gcd(4,k_1)$-periodic.
	\end{enumerate}
\end{theorem}

We next record two additional results concerning torus actions on positively curved manifolds:

\begin{theorem}[Berger]\label{Berger}
Suppose $T$ is a torus acting by isometries on an compact, positively curved manifold $M^n$. If $n$ is even, then the fixed-point set $M^T$ is nonempty, and if $n$ is odd, then a codimension one subtorus has nonempty fixed-point set.
\end{theorem}

\begin{theorem}[Maximal symmetry rank, \cite{Grove-Searle94}]\label{Grove-Searle}
If $T$ is a torus acting effectively by isometries on a compact, positively curved manifold $M^n$, then $\dim T \leq {\frac{n+1}{2}}$. Moreover, if equality holds and $M$ is simply connected, then $M$ is diffeomorphic to $S^n$ or $\C P^{n/2}$.
\end{theorem}

Finally, we use the Griesmer bound from the theory of error-correcting codes to prove the following proposition. The estimates are specifically catered to our application. The proof indicates the general bounds required.
\begin{proposition}\label{Griesmer}
Let $n\geq c\geq 2$. Assume $T^s$ acts effectively by isometries on a positively curved $n$-manifold $N^n$ with fixed point $x$. Let $\delta(n)$ be 0 if $n$ is even and 1 if $n$ is odd.
	\begin{enumerate}
	\item If 
			\[s \geq 2\log_2 n + \frac{c}{2} - 1 - \delta(n),\] then there exists an involution $\sigma\in T^s$ such that the component $N^{\sigma}_x$ of the fixed-point set of $\sigma$ that contains $x$ satisfies $\cod N^\sigma_x \equiv 0\bmod 4$ and $0<\cod N^\sigma_x \leq\frac{n-c}{2}$.
	\item Let $\sigma$ be an involution as above such that $N^\sigma_x$ has minimal positive codimension. If \[s \geq\log_2n + \frac{c}{2} + 1 + \log_2(3) - \delta(n),\] then there exists an involution $\tau\in T^s$ satisfying $N^\tau_x \not\subseteq N^\sigma_x$, $\cod N^\tau_x \equiv 0\bmod{4}$, $\cod N^\sigma_x\cap N^\tau_x \equiv 0\bmod{4}$, and $0<\cod N^\tau_x \leq\frac{n-c}{2}$.
	\end{enumerate}
\end{proposition}

By the connectedness theorem, the inclusions $N^\sigma_x\embedded N$ and
	$N^\sigma_x \cap N^\tau_x \embedded N^\tau_x \embedded N$
are $c$-connected. Since $c\geq 2$, this implies that all three submanifolds are $1$-connected. In particular, $N^\sigma_x \cap N^\tau_x = N^{\langle\sigma,\tau\rangle}_x$ where $\langle\sigma,\tau\rangle$ is the subgroup generated by $\sigma$ and $\tau$.

The only part in the proof where we use positive curvature is to conclude that, in the first statement, the bound on $s$ and the assumption $n\geq c\geq 2$ implies $n\geq 11$ by \thm{Grove-Searle}. Given this, the bound on $s$ implies
	\[s > \log_2\ceil{\frac{n-c+1}{2}} + \ceil{\frac{c}{2}} + 1.\]
We also observe that the bound in the second statement implies
	\[s > \log_2\ceil{\frac{n-c+1}{2}} + \ceil{\frac{c}{2}} + 2.\]

We proceed to the proof. We require the following algebraic lemma:

\begin{lemma}\label{lemAlgLem}
If $n\geq c\geq 2$ and $r > \ceil{\frac{c}{2}} + \log_2\ceil{\frac{n-c+1}{2}}$, then \[\floor{\frac{n}{2}} < \sum_{i=0}^{r-1} \ceil{2^{-i-1}\ceil{\frac{n-c+1}{2}}}.\]
\end{lemma}

\begin{proof}[Proof of \lem{lemAlgLem}]
We proceed by contradiction. Suppose the opposite inequality holds. The bounds $n$, $c$, and $r$ imply that $r\geq \ceil{\frac{c}{2}} + 1$ and that $\ceil{\frac{c}{2}}\geq 1$, hence we may split the sum into two pieces and estimate as follows:
	\begin{eqnarray*}
	\floor{\frac{n}{2}}
	  &\geq& \sum_{i=0}^{r-\ceil{\frac{c}{2}}-1}2^{-i-1}\ceil{\frac{n-c+1}{2}}
	        +\sum_{i=r-\ceil{\frac{c}{2}}}^{r-1} 1.
	\end{eqnarray*}
Calculating the geometric sum and rearranging, this implies
	\[\ceil{\frac{n-c+1}{2}}
	  \geq 2^{r-\ceil{\frac{c}{2}}}
		\left(\ceil{\frac{n-c+1}{2}}+\ceil{\frac{c}{2}}-\floor{\frac{n}{2}}\right).\]
Observing that the integers $n$, $c$, and $n-c+1$ cannot all be even, we conclude that the term in parentheses is at least 1, hence taking logarithms yields a contradiction to the assumed bound on $r$.
\end{proof}

We proceed to the proof of \prop{Griesmer}.
\begin{proof}
Choose a basis of $T_x M$ such that the image of $\Z_2^r\subseteq T^r$ under the isotropy representation $T^r \embedded SO(T_x M)$ lies in a copy of $\Z_2^m\subseteq T^m \subseteq SO(T_x M)$ where $m = \floor{\frac{n}{2}}$. Endow $\Z_2^m$ with a $\Z_2$-vector space structure, and consider now the linear embedding $\iota:\Z_2^r \embedded \Z_2^m$.

Consider the first statement. The bound on $s$ and the assumption $n\geq c\geq 2$ imply $s\geq 2$. Consider the composition of $\Z_2^s\to \Z_2^m\to\Z_2$, where the last map takes $\tau=(\tau_1,\ldots,\tau_m)$ to $\sum \tau_i\bmod{2}$. There exists $\Z_2^{s-1}$ inside the kernel, and for each $\sigma\in\Z_2^{s-1}$, $\cod N^\sigma_x$ is twice the even weight of $\iota(\sigma)=\tau$. Hence it suffices to prove the existence of $\sigma\in\Z_2^{s-1}\setminus\{0\}$ with $\cod N^\sigma_x\leq\frac{n-c}{2}$.

Suppose now that no such $\sigma$ exists. Then every $\sigma\in\Z_2^{s-1}\setminus\{0\}$ has $\cod N^\sigma_x \geq \ceil{\frac{n-c+1}{2}}$. Equivalently, the Hamming weight of the image of every $\sigma\in\Z_2^{s-1}\setminus\{0\}$ is at least $\frac{1}{2}\ceil{\frac{n-c+1}{2}}$. We now apply the Griesmer bound from the theory of error-correcting codes:

\begin{nonumbertheorem}[Griesmer bound, \cite{Griesmer60}]
If $\Z_2^r \embedded \Z_2^m$ is an injective linear map such that every element in the image has weight at least $w$, then
	\[m \geq \sum_{i=0}^{r-1} \ceil{\frac{w}{2^i}}.\]
\end{nonumbertheorem}

This bound implies
	\[\floor{\frac{n}{2}} \geq \sum_{i=0}^{(s-1)-1}\ceil{2^{-i-1}\ceil{\frac{n-c+1}{2}}}.\]
By the comments following the statement of \prop{Griesmer}, we have
	\[s-1 > \log_2\ceil{\frac{n-c+1}{2}} + \ceil{\frac{c}{2}},\]
hence we have a contradiction to \lem{lemAlgLem}, as desired.

We now prove the second statement of \prop{Griesmer}. First, observe that the lower bound on $s$ implies $s\geq 4$. Let $\sigma\in\Z_2^r$ be as in the statement. We define three linear maps $\Z_2^s\to \Z_2$. For the first, fix a component $i$ such that the $i$-th component of $\iota(\sigma)$ is 1 (which corresponds to a normal direction of $N^\sigma_x$), and define the map $\Z_2^s\to\Z_2$ as the projection onto the $i$-th component of $\iota(\tau)$. For the second map, assign $\tau\in\Z_2^s$ to the sum of the components of $\iota(\tau)$ (as we did previously to choose $\sigma$). And for the third map, let $I$ be the subset of indices $i$ where the $i$-th component of $\iota(\sigma)$ is 0, and define the map by assigning $\tau\in\Z_2^s$ to the sum over $i\in I$ of the $i$-th components of $\iota(\tau)$. The intersection of the kernels of these three maps contains a $\Z_2^{s-3}$, and every $\tau\in\Z_2^{s-3}$ satisfies $N^\tau_x\not\subseteq N^\sigma_x$, $\cod N^\tau_x\equiv_4 0$, and $\cod N^{\langle\sigma,\tau\rangle}_x \equiv_4 0$. Moreover, the image of every $\tau\in\Z_2^{s-3}$ has a 0 in the $i$-th component of $\Z_2^m$, hence we have an injective linear map $\Z_2^{s-3} \to \Z_2^{m-1}$ where twice the weight of the image of $\tau\in\Z_2^{r-3}$ is equal to $\cod N^\tau_x$.

It therefore suffices to prove that the element in $\Z_2^{r-3}$ whose image has minimal weight has weight at most $\frac{n-c}{2}$. Proceeding again by contradiction, we assume this is not the case and conclude from the Griesmer bound that
	\[\floor{\frac{n}{2}} - 1 \geq \sum_{i=0}^{(s-3)-1} \ceil{2^{-i-1}\ceil{\frac{n-c+1}{2}}}.\]
The bound on $s$ implies that the $i = s-3$ term in the sum would be 1, hence we may add one to both sides of this inequality to conclude
	\[\floor{\frac{n}{2}} \geq \sum_{i=0}^{(s-2)-1} \ceil{2^{-i-1}\ceil{\frac{n-c+1}{2}}}.\]
As established after the statement of \prop{Griesmer}, the bound on $s$ implies
	\[s-2 > \log_2\ceil{\frac{n-c+1}{2}} + \ceil{\frac{c}{2}},\]
so we have another contrdiction to \lem{lemAlgLem}. This concludes the proof of \prop{Griesmer}.
\end{proof}

\bigskip
\section{Proof of \thm{thmP-intro}}\label{P}
\bigskip

In this section, we use the following notation:

\begin{definition}
For integers $n$, let $\delta(n)$ be 0 if $n$ is even and $1$ if $n$ is odd. And for $c\geq 2$, let
\[f_c(n) = 2\log_2 n + \frac{c}{2} - 1 - \delta(n).\]
\end{definition}
Given an isometric action of an $r$-torus on a closed, positively curved $n$-manifold with $r\geq 2\log_2 n + \frac{c}{2} - 1$, \thm{Berger} implies that a subtorus of dimension $r-\delta(n) \geq f_c(n)$ has a fixed point. Using this, one can conclude \thm{thmP-intro} from the following:
\begin{theorem}\label{thmP}
Let $n\geq c\geq 2$. Assume $M^n$ is a closed, simply connected, positively curved manifold, and assume a torus $T$ acts effectively by isometries with $\dim(T) \geq f_c(n)$. For all $x\in M^T$, there exists $H\subseteq T$ such that $H^*(M^H_x;\Q)$ is $4$-periodic and the inclusion $M^H_x\embedded M$ is $c$-connected.

Moreover, $H$ may be chosen to satisfy $\dim M^H_x\equiv n \bmod{4}$, $\dim M^H_x \geq c$, and the property that every free group action $\pi\times M\to M$ commuting with the action of $T$ restricts to a $\pi$-action on $M^H_x$.
\end{theorem}


Here, and throughout this section, we use the notation $M^H$ to denote the fixed-point set of $H$, and we write $M^H_x$ for the component of $M^H$ containing $x$.

In the case where $M$ is not simply connected, we consider the universal cover $\tilde{M}$ of $M$. The torus action on $M$ induces an action by a torus $\tilde{T}$ of the same dimension on $\tilde{M}$. Moreover, the fundamental group $\pi = \pi_1(M)$ acts freely on $\tilde{M}$ and its action on $\tilde{M}$ commutes with the action of $\tilde{T}$. \thm{thmP} now implies that there exists a $c$-connected inclusion $N\embedded \tilde{M}$ such that $\pi$ acts (freely) on $N$, $N$ has 4-periodic rational cohomology, and $\dim N \equiv \dim M \bmod{4}$.

In the case where $\dim M \equiv 1 \bmod{4}$, 4-periodicity and Poincar\'e duality imply that $N$ is a simply connected rational homology sphere. This proves the corollary stated in the introduction. Similarly, if $\dim M \equiv 3 \bmod{4}$, the conclusion is that $\pi_1(M)$ acts freely on a simply connected rational homology sphere or a simply connected rational $S^3\times \HH P^{(n-3)/4}$.

We spend the rest of this section on the proof of \thm{thmP}. First observe that $n\geq c\geq 2$ implies
	\[f_c(n) \geq \ceil{\log_2 n} \geq \floor{\frac{n+1}{2}}\] for $2\leq n \leq 5$, hence the theorem is true in low dimensions by \thm{Grove-Searle}. Simply take $H$ to be the identity subgroup of $T$.

Since the result holds in dimensions less than six, we may proceed by induction. We spend the rest of this section proving the induction step. For this purpose, we assume the following:

	\begin{itemize}
	\item $c \geq 2$,
	\item $M$ is a closed, simply connected, positively curved $n$-manifold with $n \geq c$,
	\item $T$ is a torus acting effectively by isometries on $M$ with $\dim(T) \geq f_c(n)$, and
	\item $x$ is a fixed point in $M^T$.
	\end{itemize}

To put what we are trying to prove in a simpler light, we make the following definition:

\begin{definition} Let $c$, $M$, $T$, and $x$ be as above. 
Denote by $\mathcal{C}$ the set of $M^H_x$ where $H\subseteq T$ ranges over subgroups such that
		\begin{itemize}
		\item the inclusion $M^H_x\embedded M$ is $c$-connected,
		\item $\cod M^H_x \equiv 0 \bmod{4}$, 
		\item $\dim M^H_x \geq c$, and
		\item every free group action $\pi\times M \to M$ commuting with the action of $T$ restricts to a $\pi$-action on $M^H_x$.
		\end{itemize}


\end{definition}

Observe that our goal is to prove the following:
\smallskip
\begin{claim}
There exists $M^H_x \in \mathcal{C}$ with $4$-periodic rational cohomology.
\end{claim}
\smallskip
Our first step is to draw a conclusion from our induction hypothesis. To state the lemma, we require one more definition:

\begin{definition}
For a submanifold $N\subseteq M$ on which $T$ acts, let $\ker(T|_N)\subseteq T$ denote the kernel of the induced $T$-action on $N$. Also let $\dk(N) = \dim \ker(T|_N)$, that is, the dimension of the kernel of the induced $T$-action on $N$.
\end{definition}

Since $M$ and $T$ are fixed, the quantity $\dk(N)$ is well defined. 
We now put our induction hypothesis to use:

\begin{lemma}\label{lemI}
Some $M^H_x \in \mathcal{C}$ has $4$-periodic rational cohomology, or the following holds: For all $Q,N\in\mathcal{C}$ with $Q\subseteq N\subseteq M$ and $\dim Q < n$,
	\begin{enumerate}
	\item $\dim Q > n/2^{\dk(Q)/2}$ and
	\item if $k < n/(3\cdot 2^{\dk(N)})$, then
	\[\dim Q > \left\{\begin{array}{rcl}
			 k	&\mathrm{if}&	2\dk(N) - \dk(Q) \geq -3\\
	\frac{3}{2}k	&\mathrm{if}&	2\dk(N) - \dk(Q) \geq -2\\
			2k	&\mathrm{if}&	2\dk(N) - \dk(Q) \geq -1\\
			3k	&\mathrm{if}&	2\dk(N) - \dk(Q) \geq 0.
			\end{array}\right.\]
	\end{enumerate}
\end{lemma}
\begin{proof}
Suppose for a moment that there exists $S\in\mathcal{C}$ such that $\dim S < n$ and $\dim S \leq n/2^{\dk(S)/2}$. Then $T/\ker(T|_S)$ is a torus acting effectively on $S$ with dimension
	\[\dim(T) - \dk(S) \geq f_c(n) - 2\log_2\of{n} + 2\log_2\of{\dim S} = f_c(\dim S).\]
Since $\dim S < n$, the induction hypothesis implies the existence of an $H'\subseteq T/\ker(T|_S)$ such that
	\begin{itemize}
	\item $S^{H'}_x$ has 4-periodic rational cohomology,
	\item $S^{H'}_x\embedded S$ is $c$-connected,
	\item $\dim S^{H'}_x \equiv \dim S\bmod{4}$,
	\item $\dim S^{H'}_x \geq c$, and
	\item every free group action $\pi\times S \to S$ that commutes with the action of $T/\ker(T|_S)$ restricts to a $\pi$-action on $S^{H'}_x$.
	\end{itemize}
Letting $H$ be the inverse image of $H'$ under the quotient map $T\to T/\ker(T|_S)$, we conclude that $S^{H'}_x = M^H_x$. Moreover, since $S \in \mathcal{C}$, we have $M^H_x\in\mathcal{C}$. Hence $M^H_x \in\mathcal{C}$ and has 4-periodic rational cohomology.

We may assume therefore that no such $S$ exists. Letting $Q$ and $N$ be as in the assumption of the lemma, we immediately obtain the estimate $\dim Q > n/2^{\dk(Q)/2}$. The second estimate on $\dim Q$ follows directly from the first together with the estimate on $k$.
\end{proof}

Since our goal is to prove that some $M^H_x\in\mathcal{C}$ has 4-periodic rational cohomology, we assume from now on the second statement of \lem{lemI}.

We next begin the study of fixed-point sets of involutions. Using \prop{Griesmer}, we prove the following:

\begin{lemma}\label{lemII}
There exists an involution $\sigma\in T$ such that $M^\sigma_x\in\mathcal{C}$, $\dk(M^\sigma_x) \leq 2$, and $0 < \cod M^\sigma_x \leq \frac{n-c}{2}$.
\end{lemma}

\begin{proof}
Recall that $x\in M$ has been fixed. Also recall that $\dim(T)\geq f_c(n)$. By the first part of \prop{Griesmer}, there exists $\sigma\in T$ satisfying $\cod M^{\sigma}_x \equiv 0\bmod{4}$ and $0< \cod M^{\sigma}_x \leq \frac{n - c}{2}$. By choosing $\sigma$ among all such involutions so that $\cod M^{\sigma}_x$ is minimal, we ensure that $\dk(M^{\sigma}_x) \leq 2$. It remains to show that $M^\sigma_x = M^{\langle\sigma\rangle}_x\in\mathcal{C}$.

First, the bound on $\cod M^\sigma_x$ and the connectedness theorem implies that $M^\sigma_x\embedded M$ is $c$-connected. Second, our choice of $\sigma$ implies $\dim M^\sigma_x \equiv \dim M \bmod{4}$. Third, the bound $\cod M^\sigma_x\leq \frac{n - c}{2}$ and the fact that $n\geq c$ imply that $\dim M^\sigma_x \geq \frac{n+c}{2} \geq c$. Finally, the assumption that $c\geq 2$ implies that $\dim M^\sigma_x \geq \frac{n}{2}$. By the connectedness theorem, $M^\sigma_x$ is the unique component of $M^\sigma$ with dimension at least $n/2$, which implies that every $\pi$-action on $M$ that commutes with $T$ preserves $M^\sigma_x$ and hence restricts to a $\pi$-action on $M^\sigma_x$. These conclusions imply $M^\sigma_x\in\mathcal{C}$.
\end{proof}

Using the periodicity theorem, we can sharpen the conclusion of \lem{lemII} to the following:

\begin{lemma}\label{lemII-improvement}
Some $M^H_x\in\mathcal{C}$ has $4$-periodic rational cohomology, or there exists an involution $\sigma\in T$ such that $M^\sigma_x\in\mathcal{C}$ and $0<\cod M^\sigma_x \leq\min\left(\frac{n-c}{2}, \frac{n}{3}\right)$.
\end{lemma}

\begin{proof}
Choose $\sigma\in T$ as in the previous lemma. If, in fact, $\dk(M^\sigma_x) \leq 1$, then \lem{lemI} implies
	\[\cod M^\sigma_x < \left(n-\frac{n}{\sqrt{2}}\right)
		<\frac{n}{3}.\]
Otherwise, there exists a 2-torus inside $T$ which fixes $M^\sigma_x$. It follows that there exists another involution $\iota\in T$ such that $M^{\sigma}_x \subseteq M^\iota_x \subseteq M$ with both inclusions strict. Since $\iota$ and $\sigma$ are involutions in $T$, it follows that $M^{\sigma}_x$ is the transverse intersection of $M^\iota_x$ and $M^{\iota\sigma}_x$. Since
	\[2\cod M^\iota_x + 2\cod M^{\iota\sigma}_x 
		= 2\cod M^{\sigma}_x < n,\]
the periodicity theorem implies that $H^*(M;\Q)$ is 4-periodic. Since $M = M^{\langle\id\rangle}_x \in \mathcal{C}$, this concludes the proof.
\end{proof}

Since our goal is to prove that some $M^H_x \in \mathcal{C}$ has 4-periodic rational cohomology, this lemma allows us to assume without loss of generality that $(M,\sigma)$ satisfies \propI according to the following definition:


\begin{definition}
We say that $(N, \sigma)$ satisfies \propI if $N\in\mathcal{C}$ and $\sigma$ is an involution in $T/\ker(T|_N)$ such that $N^\sigma_x \in \mathcal{C}$ and
	\[0 < \cod_N N^\sigma_x 
		\leq \min\of{\frac{\dim N - c}{2},
		             \frac{n}{3\cdot 2^{\dk(N)}}}.\]
\end{definition}

Here and throughout the rest of the proof, $\cod_R Q = \cod Q - \cod R$ denotes the codimension of $Q\subseteq R$. As established before the definition, there exists at least one pair satisfying \propI. We focus our attention on a particular minimal pair. Specifically, among pairs $(N, \sigma)$ satisfying \propI with minimal $\dim N$, we choose one with minimal $\cod_N N^\sigma_x$.

With $N$ fixed, we will denote by $\overline{T}$ the quotient of $T$ by the kernel $\ker(T|_N)$ of the induced $T$-action on $N$. Observe that $\overline{T}$ acts effectively on $N$ and has dimension $\dim T - \dk(N)$. Moreover, we wish to emphasize that the involution $\sigma$ lies in $\overline{T}$.

The strategy for the rest of the proof is to choose a second involution in $\overline{T}$ in a certain minimal way, analyze the consequences of our minimal choices to prove \lem{lemIII} below, then to conclude the proof of \thm{thmP}.

To begin, we prove the following:

\begin{lemma} There exists an involution $\tau\in \overline{T}$ such that 
	\[\cod N^\tau_x 
	\equiv \cod N^{\langle\sigma,\tau\rangle}_x
	\equiv 0\bmod{4}\]
and
	\[0 < \cod_N N^\tau_x \leq \frac{\dim N - c}{2}.\]

Moreover, given any such $\tau$, the submanifolds $N^\tau_x$, $N^{\sigma\tau}_x$, and $N^{\langle\sigma,\tau\rangle}_x$ lie in $\mathcal{C}$.
\end{lemma}

\begin{proof}
The existence of such a $\tau$ follows from \prop{Griesmer} once we establish that
	\[\dim(\overline{T}) \geq \log_2(\dim N) + \frac{c}{2} + 1 + \log_2(3) - \delta(\dim N).\]
Moreover, we see that this is the case by combining the following facts:
	\begin{itemize}
	\item $\dim(\overline{T}) = \dim(T) - \dk(N)$ by definition of $\dk(N)$,
	\item $\dim(T) \geq f_c(n)$ by assumption,
	\item $\delta(\dim N) = \delta(n)$ because $\dim N \equiv n \bmod{4}$, and
	\item $\dk(N) \leq \log_2(n) -\log_2(3) - 2$ because $(N,\sigma)$ satisfying \propI implies that $4\leq\cod N^\sigma_x \leq n/(3\cdot 2^{\dk(N)})$.
	\end{itemize}

For the second claim, let $\tau$ be any involution satisfying satisfying these properties. We wish to show that $N^\tau_x$, $N^{\sigma\tau}_x$, and $N^{\langle\sigma,\tau\rangle}_x$ lie in $\mathcal{C}$.

First let $H\subseteq T$ be such that $N = M^H_x$. If $p:T\to \overline{T}$ is the projection map, then $N^\tau_x = M^{\langle H, p^{-1}(\tau)\rangle}_x$, where $\langle H, p^{-1}(\tau)\rangle$ is the subgroup of $T$ generated by $H$ and $p^{-1}(\tau)$. Similarly, $N^{\sigma\tau}_x = M^{\langle H, p^{-1}(\sigma\tau)\rangle}_x$ and $N^{\langle\sigma,\tau\rangle}_x = M^{\langle H, p^{-1}(\langle\sigma,\tau\rangle)\rangle}_x$. 

Second, the inclusions $N^\sigma_x \cap N^\tau_x \embedded N^\tau_x \embedded N$
are $c$-connected by the connectedness theorem together with the upper bounds on $\cod_N N^\sigma_x$ and $\cod_N N^\tau_x$. In particular, $N^\sigma_x\cap N^\tau_x$ is connected, so we have $N^{\langle\sigma,\tau\rangle}_x = N^\sigma_x \cap N^\tau_x$. Also by the connectedness theorem, the inclusion
	$N^{\langle\sigma,\tau\rangle}_x \embedded N^{\sigma\tau}_x$
is $(c+1)$-connected. Since $N\in\mathcal{C}$, this proves that the inclusions of $N^\tau_x$, $N^{\sigma\tau}_x$, and $N^{\langle\sigma,\tau\rangle}_x$ into $M$ are $c$-connected. 

Third, the dimension of all three submanifolds is congruent to $n$ modulo 4 since $\dim N \equiv n\bmod{4}$, $\cod_N N^{\sigma}_x\equiv 0\bmod{4}$, $\cod_N N^{\tau}_x\equiv 0\bmod{4}$, $\cod_N N^{\langle\sigma,\tau\rangle}_x \equiv 0\bmod{4}$, and finally
	\[\cod_N N^{\sigma\tau}_x \equiv \cod_N N^{\sigma}_x
	                                +\cod_N N^{\tau}_x
	                                \equiv 0 \bmod{4}.\]

Fourth, the dimension of all three submanifolds is at least that of $M^{\langle\sigma,\tau\rangle}_x$, which has dimension at least \[\dim N - \cod_N N^\sigma_x - \cod_N N^\tau_x\geq c.\]

Finally, let $\pi\times M \to M$ be a free group action commuting with the $T$-action on $M$. We wish to show that the $\pi$-action restricts to $\pi$-actions on $N^\tau_x$, $N^\sigma_x \cap N^\tau_x$, and $N^{\sigma\tau}_x$. This follows from the following observations:
	\begin{itemize}
	\item By assumption, $N\in\mathcal{C}$, so the $\pi$-action restricts to a $\pi$-action on $N$.
	\item By assumption, the dimensions of $N^\sigma_x$ and $N^\tau_x$ in $N$ are at least $\frac{1}{2}\dim(N)$, hence the connectedness theorem implies that $N^\sigma_x$ and $N^\tau_x$ are the unique such components of $N^\sigma$ and $N^\tau$, respectively. It follows that the $\pi$-action preserves $N^\sigma_x$ and $N^\tau_x$.
	\item Since $\pi$ preserves $N^\sigma_x$ and $N^\tau_x$, it also preserves $N^\sigma_x \cap N^\tau_x$.
	\item Since $\pi$ preserves $N^\sigma_x \cap N^\tau_x$, the fact that $N^\sigma_x \cap N^\tau_x \subseteq N^{\sigma\tau}_x$ implies that $\pi$ also preserves $N^{\sigma\tau}_x$.
	\end{itemize}
This concludes the proof that $N^\tau_x$, $N^{\sigma\tau}_x$, and $N^{\langle\sigma,\tau\rangle}_x$ satisfy the requirements needed to be elements of $\mathcal{C}$.
\end{proof}

We now choose $\tau\in\overline{T}$ such that
	\[\cod_N N^\tau_x \equiv \cod_N N^{\langle\sigma,\tau\rangle}_x \equiv 0 \bmod{4}\]
and
	\[0 < \cod_N N^\tau_x \leq \frac{\dim(N) - c}{2}\]
and such that $\cod_N N^\tau_x$ is minimal among all such choices.

Having chosen $\tau$, we use the minimality of $\dim N$ and $\cod_N N^\tau_x$ to obtain the following:

\begin{lemma}\label{lemIII}
Both of the following hold:
	\begin{enumerate}
	\item $\dk(N^{\sigma\tau}_x) - \dk(N) \geq 2$ or the intersection $N^\sigma_x \cap N^\tau_x$ is transverse in $N$.	
	\item $\dk(N^\tau_x)  - \dk(N)\leq 3$ with equality only if $N^\sigma_x\cap N^\tau_x$ is not transverse in $N$.
	\end{enumerate}
\end{lemma}

\begin{proof}
We prove the first statement by contradiction. We assume therefore that $\dk(N^{\sigma\tau}_x) \leq 1 + \dk(N)$ and that $N^\sigma_x\cap N^\tau_x$ is not transverse. The first assumption implies that $\overline{T}/\ker\of{\overline{T}|_{N^{\sigma\tau}_x}}
		= T/\ker\of{T|_{N^{\sigma\tau}_x}}$
has dimension at least $\dim(T) - \dk(N) - 1$. Let $\overline{\sigma}$ denote the image of $\sigma$ under the projection $\overline{T}\to\overline{T}/\ker\of{\overline{T}|_{N^{\sigma\tau}_x}}$, and observe that $\of{N^{\sigma\tau}_x}^{\overline{\sigma}}_x = N^{\langle\sigma,\tau\rangle}_x$. The second assumption implies that the inclusion $N^{\langle\sigma,\tau\rangle}_x \subseteq N^{\sigma\tau}_x$ has positive codimension. Moreover, this codimension is at most $\frac{1}{2}\cod_N N^\sigma_x$ by the minimality of $\cod_N N^\tau_x$. Putting these facts together, we see that $(N^{\sigma\tau}_x, \overline{\sigma})$ satisfies \propI. Since $\dim N^{\sigma\tau}_x < \dim N$, this is a contradiction to the minimality of $\dim N$.

Proceeding to the second statement, suppose for a moment that $\dk(N^\tau_x) \geq 4 + \dk(N)$. Then there exits a 4-torus inside $\overline{T}$ that fixes $N^\tau_x$. It follows that we may choose an involution $\iota$ inside this 4-torus such that
	\[\cod_N N^\iota_x \equiv \cod_N N^{\langle\sigma,\iota\rangle}_x \equiv 0 \bmod{4}\]
and $N^\tau_x \subseteq N^\iota_x \subseteq N$ with both inclusions strict. Moreover, since $N^\tau_x\not\subseteq N^\sigma_x$, it follows for free that $N^\iota_x \not\subseteq N^\sigma_x$. Hence we have a contradiction to the minimality of $\cod_N N^\tau_x$, and we may conclude that $\dk(N^\tau_x) \leq 3 + \dk(N)$.

For the equality case, suppose that $\dk(N^\tau_x) = 3 + \dk(N)$ and that $N^\sigma_x\cap N^\tau_x$ is transverse in $N$. Since a 3-torus inside $\overline{T}$ fixes $N^\tau_x$, we may choose an involution $\iota\in \overline{T}$ such that $\cod_N N^\iota_x \equiv 0 \bmod{4}$ and $N^\tau_x \subseteq N^\iota_x \subseteq N$ with all inclusions strict. Since $N^\sigma_x\cap N^\tau_x$ is transverse, it follows that $\cod_N N^{\langle\sigma,\iota\rangle}_x \equiv 0\bmod{4}$ as well, hence we have another contradiction to the minimality of $\cod_N N^\tau_x$.
\end{proof}

We are ready to conclude the proof of \thm{thmP}. We do this by breaking the proof into cases and showing in each case that $N^{\langle\sigma,\tau\rangle}_x$ has 4-periodic rational cohomology. Since we have already established that $N^{\langle\sigma,\tau\rangle}_x \in \mathcal{C}$, this would conclude the proof of \thm{thmP}. The three cases are as follows:

\begin{description}
\item[Case 1] $2\dk(N) - \dk(N^\tau_x) \leq -2$
\item[Case 2] $2\dk(N) - \dk(N^\tau_x) \geq -1$ and $N^\sigma_x\cap N^\tau_x$ is not transverse.
\item[Case 3] $2\dk(N) - \dk(N^\tau_x) \geq -1$ and $N^\sigma_x\cap N^\tau_x$ is transverse.
\end{description}

Clearly one of these cases occurs, so our task will be complete once we show, in each case, that $N^{\langle\sigma,\tau\rangle}_x$ has 4-periodic rational cohomology. We assign each case its own lemma.


\begin{lemma}[Case 1] If $2\dk(N) - \dk(N^\tau_x) \leq -2$, then $N^{\langle\sigma,\tau\rangle}_x$ has $4$-periodic rational cohomology.
\end{lemma}

\begin{proof}[Proof in Case 1]
First observe that $\dk(N^\tau_x) \geq 2 + \dk(N)$ since, by definition, $\dk(N)\geq 0$. We may therefore choose an involution $\iota\in \overline{T}$ such that $N^\tau_x \subseteq N^\iota_x \subseteq N$ with all inclusions strict. In addition, we may assume $\cod_N N^\iota_x \equiv 0\bmod{4}$ in the case where $\dk(N^\tau_x) \geq 3 + \dk(N)$. Choose a basis for the tangent space $T_x N$ so that $\sigma$, $\tau$, and $\iota$ have the following block representations

\begin{eqnarray*}
\sigma &=& \diag(-I,-I,-I,\hspace{.12in}I,\hspace{.12in}I,\hspace{.12in}I)\\
\tau   &=& \diag(-I,-I,   \hspace{.12in}I,-I,-I, \hspace{.12in}I)\\
\iota  &=& \diag(-I,\hspace{.12in}I,\hspace{.12in}I,-I,\hspace{.12in}I, \hspace{.12in}I)
\end{eqnarray*}
where the blocks of have size $b$, $a-b$, $k-a$, $m-b$, $(l-a) - (m-b)$, and $\dim N^{\langle\sigma,\tau\rangle}_x$, where $k = \cod_N N^\sigma_x$, $l = \cod_N N^\tau_x$, and $m = \cod_N N^\iota_x$, and where $a = \cod\of{N^{\langle\sigma,\tau\rangle}_x \subseteq N^{\sigma\tau}_x}$ and $b=\cod\of{N^{\langle\sigma,\iota\rangle}_x \subseteq N^{\sigma\iota}_x}$. 

Suppose for a moment that $0< b < a$. Then $(N^{\sigma\tau}_x)^\iota_x$ and $(N^{\sigma\tau}_x)^{\tau\iota}_x$ intersect transversely in $N^{\sigma\tau}_x$ and have codimensions $b$ and $a-b$, respectively. Observe that the intersection is $N^{\langle\sigma,\tau\rangle}_x$. The codimensions $b$ and $a-b$ are positive, and they satisfy
	\[2b + 2(a-b) = 2a \leq \dim N^\tau_x - k + 2a = \dim N^{\sigma\tau}_x\]
by \lems{lemI}{lemIII}. It follows from the periodicity theorem that $N^{\langle\sigma,\tau\rangle}_x$ has 4-periodic rational cohomology.

Suppose now that $b=0$ or $b=a$. It follows from \lem{lemIII} that
	\[\cod_N N^\iota_x \equiv \cod_N N^{\langle\sigma,\iota\rangle}_x \bmod{4},\]
so the minimality of $\cod_N N^\tau_x$ implies that both of these codimensions are congruent to 2 modulo 4. By our choice of $\iota$, we must have that $\dk(N^\tau_x) \leq 2 + \dk(N)$ and hence that $2\dk(N) - \dk(N^\tau_x) \geq -2$. 
Combining this with the assumption in this case, we conclude $\dk(N) = 0$ and $\dk(N^\tau_x) = 2 + \dk(N)$. The first of these equalities implies $N = M$ by \lem{lemI}. Using \lem{lemI} again, we conclude $\dim N^\tau_x \geq \frac{1}{2} \dim N$. This bound, the periodicity theorem, and the existence of the transverse intersection of $N^\iota_x$ and $N^{\iota\tau}_x$ inside $N$ imply that $H^*(N^\tau_x;\Q)$ is 4-periodic. Finally, \lem{lemI} implies
	\[\dim N^\tau_x > \frac{3}{2} k \geq k + a,\]
which implies
	\[n - k - l > \frac{1}{2}\of{n-k-l + a} = \frac{1}{2}\dim N^{\langle\sigma,\tau\rangle}_x.\]
It follows from the connectedness theorem that $N^{\langle\sigma,\tau\rangle}_x$ also has 4-periodic rational cohomology. This concludes the proof in Case 1.
\end{proof}

\smallskip

\begin{lemma}[Case 2] If $2\dk(N) - \dk(N^\tau_x) \geq -1$ and $N^\sigma_x\cap N^\tau_x$ is not transverse, then $N^{\langle\sigma,\tau\rangle}_x$ has $4$-periodic rational cohomology.
\end{lemma}

\begin{proof}[Proof in Case 2]
First observe that $\dk(N^{\sigma\tau}_x) \geq 2 + \dk(N)$ by the third statement of \lem{lemIII}, hence we may choose an involution $\iota\in \overline{T}$ such that $N^{\sigma\tau}_x \subseteq N^\iota_x \subseteq N$ with all inclusions strict. Choose a basis for the tangent space $T_x N$ so that $\sigma$, $\tau$, and $\iota$ have the following block representations

\begin{eqnarray*}
\sigma &=& \diag(-I,-I,-I,\hspace{.12in}I,\hspace{.12in}I,\hspace{.12in}I)\\
\tau   &=&\diag(-I,\hspace{.12in}I,\hspace{.12in}I,-I,-I, \hspace{.12in}I)\\
\iota&=&\diag(\hspace{.12in}I,-I,\hspace{.12in}I,-I,\hspace{.12in}I, \hspace{.12in}I)
\end{eqnarray*}
where the blocks of have size $a$, $b$, $k-a-b$, $m-b$, $(l-a) - (m-b)$, and $\dim N^{\langle\sigma,\tau\rangle}_x$. Here $k$, $l$, $m$, $a$, and $b$ have the same geometric meaning as in Case 1. The difference is hidden in the order of the blocks, which indicate that $N^{\sigma\tau}_x \subseteq N^\iota_x$ in this case while $N^\tau_x \subseteq N^\iota_x$ in Case 1.

Observe that $a>0$ because $N^\sigma_x\cap N^\tau_x$ is not transverse
. In addition, observe that the assumption in this case implies $2k+l\leq n$ by \lem{lemI}. Finally, by replacing $\iota$ by $\iota\sigma\tau$ if necessary, we may assume that $b\leq \frac{k-a}{2}$.

First suppose $b > 0$. Then $(N^\tau_x)^\iota_x$ and $(N^\tau_x)^{\sigma\iota}_x$ intersect transversely in $N^\tau_x$ with intersection $N^{\sigma\tau\iota}_x$. Since the codimensions $b$ and $k-a-b$, respectively, are positive and satisfy
	\[2b + 2(k-a-b) \leq 2k \leq n - l = \dim N^\tau_x,\]
the periodicity theorem implies that $N^{\langle\sigma,\tau\rangle}_x = (N^\tau_x)^\iota_x \cap (N^\tau_x)^{\sigma\iota}_x$ has 4-periodic rational cohomology.

Now suppose $b=0$. Then $(N^{\sigma\tau\iota}_x)^\sigma_x$ and $(N^{\sigma\tau\iota}_x)^{\sigma\tau}_x$ intersect transversely inside $N^{\sigma\tau\iota}_x$ with codimensions $a>0$ and $m$. Using the estimates $a\leq k$ and $2k+l\leq n$, it follows that
	\[3a + m \leq n - k - l + 2a + m = \dim N^{\sigma\tau\iota}_x.\]
By the periodicity theorem, $N^{\langle\sigma,\tau\rangle}_x = \of{N^{\sigma\tau\iota}_x}^\sigma_x \cap \of{N^{\sigma\tau\iota}_x}^{\sigma\tau}_x$
once again has 4-periodic rational cohomology. This concludes the proof in Case 2.
\end{proof}

\smallskip

\begin{lemma}[Case 3] If  $2\dk(N) - \dk(N^\tau_x) \geq -1$ and $N^\sigma_x\cap N^\tau_x$ is transverse, then $N^{\langle\sigma,\tau\rangle}_x$ has $4$-periodic rational cohomology.
\end{lemma}

\begin{proof}[Proof in Case 3]
Let $k = \cod_N N^\sigma_x$ and $l = \cod_N N^\tau_x$. As in the proof of Case 2, we have $2k + l \leq n$.

First we consider the case where $\dk(N^\tau_x) \geq 2 + \dk(N)$. This implies the existence of an involution $\iota\in \overline{T}$ such that $N^\tau_x \subseteq N^\iota_x \subseteq N$ with all inclusions strict. By replacing $\iota$ by $\tau\iota$ if necessary, we may assume that its codimension $m$ satisfies $m\leq \frac{l}{2}$. Since $N^\sigma_x\cap N^\tau_x$ is transverse, $N^\sigma_x\cap N^\iota_x$ is as well. Since the codimensions of this transverse intersection satisfy 
	\[2k + 2m \leq 2k + l \leq n,\]
the periodicity theorem implies that $N^{\langle\sigma,\tau\rangle}_x$ has 4-periodic rational cohomology.

Second we consider the case where 
\[\dk\of{N^{\langle\sigma,\tau\rangle}_x} \geq 2 + \dk(N^\tau_x).\] This implies the existence of an $\iota\in \overline{T}$ such that $\iota^2|_{N^\tau_x} = \id$ and such that 
\[(N^\tau_x)^\sigma_x = N^{\langle\sigma,\tau\rangle}_x \subseteq (N^\tau_x)^\iota_x \subseteq N^\tau_x\] with all inclusions strict. It follows that $(N^\tau_x)^\iota_x$ and $(N^\iota_x)^{\sigma\iota}_x$ intersect transversely in $N^\tau_x$ and have codimensions, say, $b$ and $k-b$. Since
	\[2b + 2(k-b) = 2k \leq n - l = \dim N^\tau_x,\]
the periodicity theorem implies $N^{\langle\sigma,\tau\rangle}_x$ has 4-periodic rational cohomology.

Third, we consider the case where $2\dk(N) - \dk(N^\tau_x) \geq 0$. \lem{lemI} implies $\dim N^\tau_x \geq 3k$, hence $3k+l\leq n$. Since $N^\sigma_x$ and $N^\tau_x$ intersect transversely, the periodicity theorem implies $N^{\langle\sigma,\tau\rangle}_x$ has 4-periodic rational cohomology.

Finally, if none of these three possibilities occurs, the assumption in this case implies that $\dk(N)= 0$, $\dk(N^\tau_x) = 1 + \dk(N)$, and $\dk(N^{\langle\sigma,\tau\rangle}_x) \leq 2 + \dk(N)$. Using \lem{lemI}, we can further conclude $N = M$ and $\dim N^{\langle\sigma,\tau\rangle}_x \geq \frac{1}{2}\dim N$. Hence
	\[2k + 2l = 2\cod N^{\langle\sigma,\tau\rangle}_x \leq n,\]
so the periodicity theorem applied to the transverse intersection of $N^\sigma_x$ and $N^\tau_x$ implies that $H^*(N^{\langle\sigma,\tau\rangle}_x;\Q)$ is 4-periodic. This concludes the proof of Case 3, and hence concludes the proof of \thm{thmP}.
\end{proof}

\bigskip
\section{From \thm{thmP-intro} to \thm{symspsTHM}}\label{symsps}
\bigskip

The proof of \thm{symspsTHM} contains three steps. The first step classifies 1-connected, compact, irreducible symmetric spaces whose rational cohomology either is 4-periodic up to degree $16$ or has the property that $H^i(M;\Q)=0$ for $3<i<16$. The second step is a lemma about product manifolds whose rational cohomology is 4-periodic up to degree 16. The final step combines these lemmas to classify 1-connected, compact symmetric spaces whose rational cohomology is 4-periodic up to degree 16. From \thm{thmP-intro}, these results immediately imply \thm{symspsTHM}.

The first relevant lemma concerning 1-connected, compact, irreducible symmetric spaces is the following:

\begin{lemma}\label{LEMirreducible}
Assume $M^n$ is a 1-connected, compact, irreducible symmetric space. Let $c\geq 16$ and $n\geq 16$ be integers. If $H^*(M;\Q)$ is $4$-periodic up to degree $c$, then $M$ is one of the following: $S^n$, $\C P^q$, $\HH P^{q'}$, $SO(2+q)/SO(2)\times SO(q)$, or $SO(3+q')/SO(3)\times SO(q')$ where $q = n/2$ and $q'=n/3$.
\end{lemma}

Observe that periodicity up to degree $c\geq 16$ implies that $M$ either is rationally $(c-1)$-connected or has the property that $H^{4i}(M;\Q)\cong \Q$ for $4i\leq c$. Since $c\geq 16$ we may as well assume in the lemma that $n\geq 16$.

Before proving \lem{LEMirreducible}, we wish to state the other lemma that concerns irreducible symmeric spaces:

\begin{lemma}\label{LEMirreducible2}
Assume $M^n$ is a 1-connected, compact, irreducible symmetric space. If $H^i(M;\Q) = 0$ for $3 < i < 16$, then $M$ is $S^2$ or $S^3$.
\end{lemma}

The only facts about 4-periodicity up to degree 16 that we will use in the proofs are the following: $b_i = b_{i+4}$ for $0 < i < 12$, $b_4 \leq 1$, and $b_4 = 0$ only if $b_i=0$ for all $0<i<16$. Here and throughout the section, $b_i$ denotes the $i$-th Betti number of $M$.

\begin{proof}[Proof of \lem{LEMirreducible}]
We use Cartan's classification of simply connected, irreducible compact symmetric spaces. We also keep Cartan's notation. See \cite{Helgason01} for a reference.

One possibility is that $M$ is a simple Lie group. The rational cohomology of $M$ is therefore that of a product of spheres $S^{n_1}\times S^{n_2}\times\cdots\times S^{n_s}$ for some $s\geq 1$ where the $n_i$ are odd. In fact, these sphere dimensions are known and are listed in \TABLE{TABLEspheres} on page \pageref{TABLEspheres} (see \cite{Mimura-Toda78} for a reference). Since $M$ is simply connected, we may assume
	\[3 = n_1 \leq n_2 \leq \cdots\leq n_s.\]
Since $b_3(M)\neq 0$, we have $b_4(M)=1$ by our comments above. But this cannot be since the $n_i$ are odd, so there are no simple Lie groups with 4-periodic rational cohomology up to degree 16.

\begin{table}[h]\begin{center}\begin{tabular}{| c | l |}\hline
$G$ 			& $n_1, n_2, \ldots, n_s$\\\hline\hline
$Sp(n)$ 		& $3, 7, \ldots, 4n-1$\\
$Spin(2n+1)$	& $3, 7, \ldots, 4n-1$\\
$Spin(2n)$	& $3, 7, \ldots, 4n-5, 2n-1$\\
$U(n)$		& $1, 3, \ldots, 2n-1$\\
$SU(n)$		& $3, 5, \ldots, 2n-1$\\
$G_2$		& $3, 11$\\
$F_4$		& $3, 11, 15, 23$\\
$E_6$		& $3, 9, 11, 15, 17, 23$\\
$E_7$		& $3, 11, 15, 19, 23, 27, 35$\\
$E_8$		& $3, 15, 23, 27, 35, 39, 47, 59$\\\hline
\end{tabular}\caption{Dimensions of spheres}\label{TABLEspheres}\end{center}\end{table}

We now consider the irreducible spaces which are not Lie groups. We have that $M = G/H$ for some compact Lie groups $G$ and $H$ where $G$ is simple. The possible pairs $(G,H)$ fall into one of seven classical families or are one of 12 exceptional examples. We calculate the first 15 Betti numbers in each case, then compare the results to the requirement that they be 4-periodic as described above. We summarize the results in \TABLES{TABLEclassical}{TABLEexceptional}.

\begin{table}[h]
\begin{center}
\begin{tabular}{| l | c | c | c |}\hline
  $G/H$			& \begin{tabular}{c}$P_{G/H}(t)-1$ if\\
  			  	$rk(G)=rk(H)$\end{tabular}		
  			  	& \begin{tabular}{c}Reference\\if not\end{tabular}
  			  	&Obstruction\\\hline\hline 
  $SU(n)/SO(n), ~ ~ ~ n\geq 6$			& $-$		&\cite{Borel53}			& $b_5 > 0$ \\

  $SU(2n)/Sp(n), ~ ~ ~ n\geq 4$		& $-$		&\cite{Borel53}			& $b_5 > 0$ \\\hline

  $SO(p+q)/SO(p)\times SO(q)$		& $-$		&\cite{Ramani-Sankaran97}	& $1<b_4$	  \\


$SU(p+q)/S(U(p)\times U(q))$	& $2t^4 + \ldots$		& $-$			& $1<b_4$	  \\


  $Sp(n)/U(n), ~ ~ ~ n\geq 4$ 			& $t^2+t^4+2t^6+\ldots$	& $-$			& $b_2<b_6$ \\

  $Sp(p+q)/Sp(p)\times Sp(q)$		& $t^4+2t^8+\ldots$		& $-$			& $b_4<b_8$ \\


  $SO(2n)/U(n), ~ ~ ~ n\geq 5$ 		& $t^2+t^4+2t^6+\ldots$	& $-$			& $b_2<b_6$ \\\hline
\end{tabular}
\caption{Classical simply connected, compact, irreducible symmetric spaces of dimension at least 16 that are not listed in the conclusion of \lem{LEMirreducible}. The pairs $(p,q)$ satisfy $4\leq p\leq q$ for the real Grassmannians and $2\leq p\leq q$ for the complex and quaternionic Grassmannians.}\label{TABLEclassical}\end{center}\end{table}

\begin{table}\begin{center}\begin{tabular}{| l | c | c | c |}\hline
  $G/H$			& \begin{tabular}{c}$P_{G/H}(t)-1$ if \\ $rk(G)=rk(H)$\\\end{tabular}
  				& \begin{tabular}{c}Reference   \\ if not       \\\end{tabular}
  				& Obstruction \\\hline\hline
  $E_6/Sp(4)$ 			  & $-$			&\cite{Ishitoya77EI}	& $0  <b_9$\\
  $E_6/F_4			$ & $-$			&\cite{Araki-Shikata61}	& $0  <b_9$\\\hline
  $E_6/SU(6)\times SU(2)	$ & $t^4+t^6+2t^8$			& $-$ 		& $b_4<b_8$\\
  $E_6/SO(10)\times SO(2)$ & $t^4 + 2t^8 + \ldots$	& $-$ 		& $b_4<b_8$	\\
  $E_7/SU(8)			$ & $t^8 +\ldots$			& $-$ 		& $b_4<b_8$	\\
  $E_7/SO(12)\times SU(2)$ & $t^4 + 2t^8 + \ldots$	& $-$ 		& $b_4<b_8$	\\
  $E_7/E_6\times SO(2)	$ & $t^4+t^8+2t^{12}+\ldots$	& $-$ 		& $b_8<b_{12}$	\\
  $E_8/SO(16)			$ & $t^8 + \ldots$			& $-$ 		& $b_4<b_8$	\\
  $E_8/E_7\times SU(2)	$ & $t^4+t^8+2t^{12}+\ldots$	& $-$ 		& $b_8<b_{12}$	\\
  $F_4/Sp(3)\times SU(2)	$ & $t^4 + 2t^8 + \ldots$	& $-$ 		& $b_4<b_8$	\\
  $F_4/Spin(9)			$ & $t^8 + \ldots$			& $-$ 		& $b_4<b_8$	\\
  $G_2/SO(4)			$ & $t^4 + t^8$			& $-$ 		& $b_8>b_{12}$\\\hline
\end{tabular}\caption{Exceptional irreducible simply connected compact symmetric spaces not listed in \lem{LEMirreducible}}\label{TABLEexceptional}\end{center}\end{table}

To explain our calculations, we first consider $M = G/H$ with $\rank(G) = \rank(H)$. Let $S^{n_1}\times\cdots\times S^{n_s}$ and $S^{m_1}\times\cdots\times S^{m_s}$ denote the rational homotopy types of $G$ and $H$, respectively. Then one has the following formula for the Poincar\'e polynomial of $M$ (see \cite{Borel53}):
	\[P_M(t)=\sum_{i\geq 0} b_i(M) t^i = \frac{(1-t^{n_1+1})\cdots(1-t^{n_s+1})}{(1-t^{m_1+1})\cdots(1-t^{m_s+1})}.\]
For each simple Lie group $G$, the dimensions of the spheres are listed in \TABLE{TABLEspheres}. When $\rank(G) = \rank(H)$, we compute the Poincar\'e polynomial of $M$ and list the relevant terms in Tables \ref{TABLEclassical} and \ref{TABLEexceptional}.

In the case where $\rank(G)\neq \rank(H)$, we simply cite a source where the cohomology has been calculated. The tables give the pair $(G,H)$ realizing the space, the first few terms of the Poincar\'e polynomial if $\rank(G)=\rank(H)$, and the the relevant Betti number inequalities that show $M$ is not rationally 4-periodic up to degree 16.

In Table \ref{TABLEclassical}, we exclude spaces with dimension less than 16. We also exclude the rank one Grassmannians and the rank two and rank three real Grassmannians, as these are the spaces that appear in the conclusion of \thm{symspsTHM}.
\end{proof} 

\begin{proof}[Proof of \lem{LEMirreducible2}]
We are given a 1-connected, compact, irreducible symmetric space $M$ with Betti numbers that satisfy $b_i = 0$ for $3<i<16$.

First, if $\dim M \leq 3$, then $M$ is $S^2$ or $S^3$ since $M$ is simply connected. Second, if $\dim M > 3$, then the fact that $b_{\dim M} = 1$ implies that $\dim M \geq 16$. It is apparent from Table \ref{TABLEspheres} that $M$ is not a simple Lie group. On the other hand, it is apparent that $M$ cannot appear in Table \ref{TABLEclassical} or Table \ref{TABLEexceptional}. Finally, $M$ cannot be a rank one symmetric space or a rank $p$ Grassmannian $SO(p+q)/SO(p)\times SO(q)$ with $p\in\{2,3\}$ since each of these spaces has fourth Betti number equal to one.
\end{proof}

With the first step complete, we prove the following lemma about general products $M=M'\times M''$ whose rational cohomology ring is 4-periodic up to degree $c$. For simplicity we denote the Betti numbers of $M$, $M'$, and $M''$ by $b_i$, $b_i'$, and $b_i''$, respectively.

\begin{lemma}\label{LEMPeriodicProductManifolds}
Suppose $b_1 = 0$ and that $H^*(M;\Q)$ is $4$-periodic up to degree $c$ with $c\geq 9$. Suppose that $M=M'\times M''$ with $b_4'\geq b_4''$. Then either $M$ is rationally $(c-1)$-connected or $b_4'=1$ and the following hold:
	\begin{enumerate}
	\item $H^*(M';\Q)$ is $4$-periodic up to degree $c$,
	\item $b_i''=0$ for $3 < i < c$, and
	\item if $b_2'>0$ or $b_3' > 0$, then $b_2''=b_3''=0$.
	\end{enumerate}
\end{lemma}

\begin{proof}[Proof of lemma]
Let $x\in H^4(M;\Q)$ be an element inducing periodicity. If $x=0$, then $c\geq 8$ implies $M$ is rationally $(c-1)$-connected. Assume therefore that $b_4(M)=1$ (i.e., that $x\neq 0$). 

We first claim that $b_4' = 1$. Suppose instead that $0=b_4' = b_4''$. The K\"unneth theorem implies $1=b_4 = b_2'b_2''$, and hence $b_2'=b_2''=1$. Using periodicity and the K\"unneth theorem again, we have
	\[0 = b_1 = b_5 = b_5' + b_5'' + b_3' + b_3'',\]
and hence that all four terms on the right-hand side are zero. Similarly, we have
	\[2 = b_2' + b_2'' = b_2 = b_6 = b_6' + b_6''.\]
Finally, we obtain
	\[1 = b_4 = b_8 \geq b_2'b_6'' + b_6'b_2'' = b_6'+b_6'' = 2,\]
a contradiction. Assume therefore that $b_4' = 1$ and hence that $b_4'' = b_2'b_2''=0$.

Let $p:M\to M'$ be the projection map. It follows from $b_4' = b_4 = 1$ and the K\"unneth theorem that
	\[H^4(M') \cong H^4(M')\tensor H^0(M'') \embedded \bigoplus_{i+j=4}H^i(M')\tensor H^j(M'') \stackrel{\times}{\longrightarrow} H^4(M)\]
is an isomorphism. Choose $\bar{x}\in H^4(M';\Q)$ with $p^*(\bar{x}) = x$. We claim that $\bar{x}$ induces periodicity in $H^*(M')$ up to degree $c$.

First, $b_4'=1$ implies that multiplication by $\bar{x}$ induces a surjection $H^0(M')\to H^{4}(M')$. Second, consider the commutative diagram
	\[\displaystyle
	\begin{array}{ccccc}
	H^i(M')	   &\embedded& \bigoplus_{i'+i''=i} H^{i'}(M')\tensor H^{i''}(M'') 	&\stackrel{\times}{\longrightarrow}& H^i(M)\\
	\downarrow   &~	   &\downarrow										&~							&\downarrow\\
	H^{i+4}(M')  &\embedded& \bigoplus_{i'+i''=i+4} H^{i'}(M')\tensor H^{i''}(M'') 	&\stackrel{\times}{\longrightarrow}& H^{i+4}(M)
	\end{array}
	\]
where the vertical arrows from left to right are given by multiplication by $\bar{x}$, $\bar{x}\tensor 1$, and $x$, respectively. Because mutliplication by $x$ is injective for $0 < i \leq c-4$, it follows that multiplication by $\bar{x}$ is injective in these degrees as well. It therefore suffices to check that multiplication by $\bar{x}$ is surjective for $0<i<c-4$. We accomplish this by a dimension counting argument. Specifically, we claim $b_i' = b_{i+4}'$ for $0<i<c-4$. Indeed, for all $0\leq i<c-4$,  we have from periodicity, the K\"unneth theorem, and injectivity of multiplication by $\bar{x}$ the following estimate:
	\[\sum_{j=0}^i b_{i-j}'b_j'' = b_i = b_{i+4} \geq b_{i+4}'' + \sum_{j=0}^i b_{i+4-j}'b_j'' \geq b_{i+4}''+\sum_{j=0}^i b_{i-j}'b_j''.\]
Equality must hold everywhere, proving $b_i'=b_{i+4}'$ and $b_{i+4}''=0$ for all $0\leq i<c-4$. This completes the proof of the first part, as well as the second part, of the lemma.

Finally, suppose that $b_2'>0$ or $b_3'>0$. Then
	\[b_4' + (b_2' + b_3')b_2'' \leq b_4 + b_5 = 1+b_1 = b_4'\]
implies $b_2''=0$, and
	\[b_6'+(b_2' + b_3')b_3'' \leq b_5 + b_6 = b_2 = b_2' = b_6'\]
implies $b_3'' = 0$.

\end{proof}


Finally, we come to the proof of \thm{symspsTHM}. In fact, we prove the following stronger theorem:

\begin{theorem}\label{thmSymSps}
Suppose $M^n$ has the rational cohomology ring of a simply connected, compact symmetric space $N$. Let $c\geq 16$, and assume $M$ has a metric with positive sectional curvature and symmetry rank at least $2\log_2 n + \frac{c}{2} - 1$. Then there exists a (possibly trivial) product $S$ of spheres, each of dimension at least $c$, such that one of the following holds:
	\begin{enumerate}
	\item $N=S$,
	\item $N = S \times R$ with $R\in\{\C P^q, SO(2+q)/SO(2)\times SO(q)\}$, or
	\item $N = S \times R \times Q$ with $R \in \{\HH P^q, SO(3+q)/SO(3)\times SO(q)\}$ and $Q\in\{*,S^2,S^3\}$. 
	\end{enumerate}
\end{theorem}

Consider the special case where $N$ is a product of spheres, $S^k \times S^{n-k}$. This theorem implies that each sphere has dimension at least $c$, which is at least 16, so $N$ cannot be $S^k\times S^{n-k}$ with $1<k<16$, as claimed in the introduction.



Observe that the Lie group $E_8$ has the rational cohomology of a product of spheres in dimensions $3, 15, \ldots$. It follows that the rational cohomology of $E_8 \times \HH P^3$ is 4-periodic up to degree 15, so we must take $c\geq 16$ in the statement of this theorem.



\begin{proof}
Let $N^n$ be as in the theorem. 
Assuming without loss of generality that $n>0$, \lem{Grove-Searle} implies $n\geq 16$. 

By the corollary to \thm{thmP-intro}, we conclude that $N$ is rationally 4-periodic up to degree $c$. Write $N = N_1\times\cdots\times N_t$ where the $N_i$ are irreducible symmetric spaces and $b_4(N_1)\geq b_4(N_i)$ for all $i$. By \lem{LEMPeriodicProductManifolds}, $b_4(N_1)=0$ implies $N$ is rationally $(c-1)$-connected and hence is a product of spheres since $c\geq 16$. Assume therefore that $b_4(N_1)=b_4(N)=1$.

By the same lemma, $N_1$ is 4-periodic up to degree $c$. Observe that periodicity implies $n\geq 16$ since $x^4\neq 0$ where $x\in H^4(N_1;\Q)$ is an element inducing periodicity up to degree 16. By \lem{LEMirreducible}, we have that $N_1$ is $\C P^m$, $SO(2+q)/SO(2)\times SO(q)$, $\HH P^m$, or $SO(3+q)/SO(3)\times SO(q)$. We also have that $N_i$ for $i>1$ has $b_j(N_i)=0$ for all $3 < j < c$. If $N$ is $\C P^q$ or $SO(2+q)/SO(2)\times SO(q)$, then $b_2(N_1) + b_3(N_1) > 0$, which by \lems{LEMPeriodicProductManifolds}{LEMirreducible2} implies $N_i$ is a sphere for all $i>1$. This completes the proof in this case. Suppose therefore that $N_1$ is $\HH P^q$ or $SO(3+q)/SO(3)\times SO(q)$. If $b_2(N_i)+b_3(N_i)=0$ for all $i>0$, then once again we have that each $N_i$ with $i>1$ is a sphere. Suppose therefore that some $N_i$, say with $i=2$, has $b_2(N_2) + b_3(N_2) > 0$. Taking $M' = N_1\times N_2$ and $M''=N_3\times\cdots\times N_t$ in the lemma implies $N_i$ is a sphere for all $i>2$. Finally, $N_2$ is an irreducible symmetric space with $b_2(N_2) + b_3(N_2) > 0$ and $b_i(N_2)=0$ for all $3 < i < 16$. \lem{LEMirreducible2} implies that $N_2$ is a 2-sphere or a 3-sphere, completing the proof.
\end{proof}

\bibliographystyle{amsplain}
\bibliography{myrefs}

\providecommand{\bysame}{\leavevmode\hbox to3em{\hrulefill}\thinspace}
\providecommand{\MR}{\relax\ifhmode\unskip\space\fi MR }
\providecommand{\MRhref}[2]{%
  \href{http://www.ams.org/mathscinet-getitem?mr=#1}{#2}
}
\providecommand{\href}[2]{#2}
\begin{thebibliography}{10}

\bibitem{Araki-Shikata61}
S.~Arakai and Y.~Shikata, \emph{{Cohomology mod 2 of the compact exceptional
  group $E\_8$}}, Proc. Japan Acad. \textbf{37} (1961), no.~10, 619--622.

\bibitem{Bazaikin99}
Y.V. Bazaikin, \emph{{A Manifold with Positive Sectional Curvature and
  Fundamental Group $\Z\_3\oplus\Z\_3$}}, Siberian Mathematical Journal
  \textbf{40} (1999), 834--836.

\bibitem{Borel53}
A.~Borel, \emph{{Sur la cohomologie des espaces fibres principaux et des
  espaces homogenes de groupes de Lie compacts}}, The Annals of Mathematics
  \textbf{57} (1953), no.~1, 115--207.

\bibitem{Cheeger73}
J.~Cheeger, \emph{{Some examples of manifolds of nonnegative curvature}},
  Journal of Differential Geometry \textbf{8} (1973), 623�--628.

\bibitem{Frank-Rong-Wang??pre}
P.~Frank, X.~Rong, and Y.~Wang, \emph{{Fundamental groups of positively curved
  manifolds with symmetry}}, Preprint.

\bibitem{Griesmer60}
J.H. Griesmer, \emph{{A bound for error-correcting codes}}, IBM Journal of
  Research and Development \textbf{4} (1960), 532--542.

\bibitem{Grove-Searle94}
K.~Grove and C.~Searle, \emph{{Positively curved manifolds with maximal
  symmetry-rank}}, Journal of Pure and Applied Algebra \textbf{91} (1994),
  no.~1, 137--142.

\bibitem{Grove-Shankar00}
K.~Grove and K.~Shankar, \emph{{Rank two fundamental groups of positively
  curved manifolds}}, Journal of Geometric Analysis \textbf{10} (2000), no.~4,
  679--682.

\bibitem{Helgason01}
S.~Helgason, \emph{{Differential geometry, Lie groups, and symmetric spaces}},
  American Mathematical Society, 2001.

\bibitem{Ishitoya77EI}
K.~Ishitoya, \emph{{Cohomology of the symmetric space EI}}, Proc. Japan Acad.
  Ser. A Math. Sci. \textbf{53} (1977), no.~2, 56--60.

\bibitem{Kennard12pre}
L.~Kennard, \emph{{On the Hopf conjecture with symmetry}}, Preprint,
  arXiv:1203.3808 [math.DG].

\bibitem{Mimura-Toda78}
M.~Mimura and H.~Toda, \emph{{Topology of Lie Groups, I and II}}, American
  Mathematical Society, 1978.

\bibitem{Ramani-Sankaran97}
V.~Ramani and P.~Sankaran, \emph{{On degrees of maps between Grassmannians}},
  Proceedings Mathematical Sciences \textbf{107} (1997), no.~1, 13--19.

\bibitem{Rong02}
X.~Rong, \emph{{Positively curved manifolds with almost maximal symmetry
  rank}}, Geometriae Dedicata \textbf{95} (2002), no.~1, 157--182.

\bibitem{Shankar98}
K.~Shankar, \emph{{On the fundamental groups of positively curved manifolds}},
  Journal of Differential Geometry \textbf{49} (1998), no.~1, 179--182.

\bibitem{Wang07}
Y.~Wang, \emph{{On cyclic fundamental groups of closed positively curved
  manifolds}}, JP Journal of Geometry and Topology \textbf{7} (2007), no.~2,
  283--307.

\bibitem{Wilking03}
B.~Wilking, \emph{{Torus actions on manifolds of positive sectional
  curvature}}, Acta mathematica \textbf{191} (2003), no.~2, 259--297.

\end{thebibliography}
\end{document}